\documentclass[12pt]{amsart}

\usepackage{latexsym,amsmath,amssymb,amsthm,amsfonts}

\usepackage[numeric]{amsrefs}

\usepackage{listings}
\usepackage{color}

\usepackage[capitalise]{cleveref}

\usepackage{multido,pst-plot,pstricks,pst-node}
\newcommand{\R}{{\mathbb{R}}}

\renewcommand{\S}{{\mathbb{S}}}

\newcommand{\E}{{\mathbb{E}}}

\newcommand{\T}{{\mathcal{T}}}

\newcommand{\conv}{{\rm conv}}

\newcommand{\ext}{{\rm ext}}

\renewcommand{\phi}{{\boldsymbol{\varphi}}}

\newtheorem{theorem}{Theorem}%[section]
\newtheorem{lemma}{Lemma}

\theoremstyle{remark}

\title{Spherical coverings and X-raying convex bodies of constant width}

\author{A.\ Bondarenko}
\address{Department of Mathematical Sciences, Norwegian University of Science and Technology, NO-7491 Trondheim, Norway}
\email{andriybond@gmail.com}
\thanks{The first author was supported in part by Grant 275113 of the Research Council of Norway.}% This work was initiated in May 2019 while the second author visited the first author in NTNU.}

\author{A.\ Prymak}
\address{Department of Mathematics, University of Manitoba, Winnipeg, MB, R3T 2N2, Canada}
\email{prymak@gmail.com}
\thanks{The second author was supported by NSERC of Canada Discovery Grant RGPIN-2020-05357.}
\author{D.\ Radchenko}
\address{ETH Zurich, Mathematics Department, Zurich 8092, Switzerland}
\email{danradchenko@gmail.com}



\begin{document}

	\begin{abstract}
		 K.~Bezdek and Gy.~Kiss showed that existence of origin-symmetric coverings of unit sphere in $\E^n$ by at most $2^n$ congruent spherical caps with radius not exceeding $\arccos\sqrt{\frac{n-1}{2n}}$ implies the $X$-ray conjecture and the illumination conjecture for convex bodies of constant width in $\E^n$, and constructed such coverings for $4\le n\le 6$. Here we give such constructions with fewer than $2^n$ caps for $5\le n\le 15$. 
		 
		 For the illumination number of any convex body of constant width in $\E^n$, O.~Schramm proved an upper estimate with exponential growth of order $(3/2)^{n/2}$. In particular, that estimate is less than $3\cdot 2^{n-2}$ for $n\ge 16$, confirming the above mentioned conjectures for the class of convex bodies of constant width. Thus, our result settles the outstanding cases $7\le n\le 15$. 
		 
		 We also show how to calculate the covering radius of a given discrete point set on the sphere efficiently on a computer.
	\end{abstract}

\maketitle

\section{Introduction}

The problem of packing congruent spherical caps on a sphere has received considerable attention since the centers of the caps form spherical codes which have many applications~\cite{Co-Sl}. The corresponding covering problem is not that well studied. The general results of Rogers~\cite{Ro1, Ro2} have been improved in this context by B\"or\"oczky and Wintsche~\cite{Bo-Wi} and later by Dumer~\cite{Du} and Naszodi~\cite{Na}. All these results specifically target higher dimensions, and under-perform in the lower dimensions compared to concrete constructions of covering sets derived from lattices or from other regular or symmetric arrangements of points. Motivated by applications in certain problems from convex geometry considered by Bezdek and Kiss~\cite{Be-Ki}, our goal in this work is to construct several spherical coverings with some additional properties such as origin-symmetry and a specific covering radius. Our constructions, as well as the method for calculation of the covering radius, may also be of independent interest. Now let us describe the corresponding geometric problems.

A convex body in the $n$-dimensional Euclidean space $\E^n$ is a convex compact set with non-empty interior. A point $x$ on the boundary of a convex body $K$ in $\E^n$ is illuminated along a direction $\xi\in\S^{n-1}$ (where $\S^{n-1}$ is the unit sphere in $\E^n$) if the ray $\{x+t\xi:t\ge0 \}$ intersects the interior of $K$. A convex body $K$ is illuminated along a set of directions $E\subset\S^{n-1}$ if for any point of the boundary of $K$ there is a direction $\xi\in E$ such that this point is illuminated by $\xi$. The illumination number $I(K)$ is defined as the smallest cardinality of a set of directions illuminating $K$. The well-known illumination conjecture is that for any convex body $K\subset\E^n$ one has $I(K)\le 2^n$. Note that the illumination number of an $n$-cube is $2^n$. An equivalent formulation of the illumination conjecture is that any convex body $K\subset\E^n$ can be covered by at most $2^n$ smaller homothetic copies of $K$. For a survey on these conjectures, also known as (Levi-)Hadwiger conjecture or Gohberhg-Markus covering conjecture, see~\cite{Be-Kh} and references therein; for recent results in the asymptotic case see~\cite{HSTV}; for recent results in the low-dimensional case see~\cite{Pr-Sh}; for a computer-based approach see~\cite{Zong}.

A related concept to illumination is that of X-raying a convex body introduced by Soltan. A point $x\in K$, where $K\subset\E^n$ is a convex body, is X-rayed along a direction $\xi\in\S^{n-1}$ if the line $\{x+t\xi: t\in\R \}$ intersects the interior of $K$. $K$ is X-rayed by $E\subset\S^{n-1}$ if for every point $x\in K$ there is a direction $\xi\in E$ such that $x$ is X-rayed along $\xi$. The X-ray number $X(K)$ is the smallest cardinality of a set of directions X-raying $K$. X-raying conjecture by Bezdek and Zamfirescu is that $X(K)\le 3\cdot 2^{n-2} $ for any convex body $K\subset\E^n$. An example achieving the bound is the convex hull of the vertices of an $n$-cube with one $(n-2)$-dimensional face removed. The reader can refer to~\cite{Be-Ki} for further details. 

The connection between the X-raying and the illumination problems is not hard to observe: one always has $X(K)\le I(K)\le 2X(K)$ for any convex body $K$.

Convex body has constant width if its projection onto any line has length independent of the choice of the direction of the line. This class of convex bodies plays a very important role in convex geometry and other areas of mathematics, see, e.g. \cite{constwidthbook} for a comprehensive exposition. We define $X_n^w$ and $I_n^w$ as the largest values of $X(W)$ and $I(W)$, respectively, where $W$ varies over all convex bodies of constant width in $\E^n$. A natural problem considered by Bezdek and Kiss in~\cite{Be-Ki} is to confirm X-raying and illumination conjectures for the class of convex bodies of constant width, e.g. to establish $X_n^w\le 2^{n-1}$. 

Using an interesting probabilistic argument, O.~Schramm proved in~\cite{Sc} that asymptotically $I_n^w < n^{1.5+o(1)} (3/2)^{n/2}$ as $n\to\infty$. He provided an explicit estimate for all $n$, namely, (see~\cite{Sc}*{p.~188})
\begin{equation}\label{eqn:sc}
	I_n^w < 1+4n\sqrt{\pi n/3} \ln(13+16n) \left(\frac32\right)^{n/2}.
\end{equation}
If $n\ge 16$, the right-hand-side of~\eqref{eqn:sc} is less than $3\cdot 2^{n-2}$ and we always have $X_n^w\le I_n^w$, so~\eqref{eqn:sc} confirms the X-raying and illumination conjectures for the class of convex bodies of constant width and dimensions $n\ge 16$. (We remark that the simpler estimate than~\eqref{eqn:sc} given in \cite{Sc}*{Th.~1} is not sufficient for $n=16$, and, on the other hand, further fine-tuning of parameters and constants in the proof in~\cite{Sc} does not seem to allow to confirm the conjectures of our interest for $n=15$.)

Returning to small dimensions, for $n\le 6$ the inequality $X_n^w\le 2^{n-1}$ was confirmed by Bezdek and Kiss in~\cite{Be-Ki} by reduction to a specific covering problem on the sphere. Let us explicitly formulate this reduction which is valid in all dimensions. For a finite set $A\subset\S^{n-1}$, the covering radius of $A$ is the smallest $r>0$ such that the union of spherical balls of radii $r$ centered at the points of $A$ is $\S^{n-1}$. $A$ is origin-symmetric if $-A=A$. Let $w_n$ denote the smallest cardinality of an origin-symmetric set $A\subset\S^{n-1}$ with covering radius not exceeding $\arccos\sqrt{\tfrac{n-1}{2n}}$.

\begin{lemma}[\cite{Be-Ki}*{Lemma~3.1}]\label{lem:reduction}
	$X^w_n\le \frac12 w_n$.
\end{lemma}

Let us briefly describe two main ingredients in the proof of this lemma; full details can be found in~\cite{Be-Ki}. The key concept is that of the Gauss image of a face (intersection of the boundary with a supporting hyperplane) of a convex body, which is the set of outer unit normal vectors of all supporting hyperplanes containing the face, see also~\cite{constwidthbook}*{p.~35} for the simpler case of smooth boundary. If the Gauss image (which is a subset of the unit sphere) of any face of a convex body can be covered by an appropriate spherical cap, then an estimate on the $X$-ray number of the body follows, as established in~\cite{Be-Ki}*{Lemma~2.4}. The second ingredient, which was used in~\cite{Sc} as well, is a nice geometric property of convex bodies of constant width stating that the angle between any two outer unit normal vectors of supporting hyperplanes at the same point of the boundary is at most $\pi/3$. The value $\arccos\sqrt{\tfrac{n-1}{2n}}$ arises as the complementary angle to the circumradius of a regular $(n-1)$-dimensional spherical simplex of edge length $\pi/3$.

It was shown in~\cite{Be-Ki} that $w_4\le 12$ and $w_n\le 2^n$ for $n=5,6$ and was asked if this inequality can be extended to $7\le n\le 15$. We show by explicit construction that $w_n< 2^n$ for $5\le n\le 15$, and thus completely confirm the X-raying and the illumination conjectures for the class of convex bodies of constant width in any dimension. X-raying problem is connected to a theorem of Danzer and Gr\"unbaum~\cite{Da} on antipodal convex polytopes, see~\cite{Be-Ki}*{Sec.~4}. Further, X-raying problem has found applications in approximation theory~\cite{Da-Pr}*{Sec.~7} where explicit upper bounds on the number of directions required for X-raying are of interest.

Our main result is the following theorem.

\begin{theorem}\label{thm:main}
	$w_5\le 30$, $w_6\le44$, $w_7\le112$, $w_8\le 240$, $w_9\le 470$, $w_{10}\le 692$, $w_{11}\le 2024$, $w_{12}\le 3832$, $w_{13}\le 7074$, $w_{14}\le 11132$, $w_{15}\le 16442$.  
\end{theorem}

Our constructions started from an observation that the (normalized) minimal norm vectors of the  $E_8$ lattice (see, e.g.~\cite{Co-Sl}*{Sect.~4.8.1, p.~120}) settle the problem for $n=8$. We further explored various origin-symmetric systems of vectors which are invariant under permutations of coordinates and were able to solve the problem for the outstanding dimensions and also improve the known results from~\cite{Be-Ki} for $n=5,6$. 

An important part of the proof which can be of independent interest is an efficient procedure for computation of covering radius of a given point system which is based on the computation of the polar of a convex polytope, see \cref{sec:numerical computation}. SageMath (\cite{sagemath}) code we used for computations can be found in the Appendix of the pre-print~\cite{self-pre-print} of this article.

\section{Computation of covering radius}\label{sec:numerical computation}

Recall that the polar of a convex body (convex compact set with non-empty interior) $K\subset\E^n$ containing the origin is defined as $K^\circ=\{x\in\E^n: \langle x, y\rangle\le1\ \forall y\in K \}$, where $\langle \cdot,\cdot\rangle$ is the canonical Euclidean scalar product. By $\conv(A)$ we denote the convex hull of $A$, by $\|\cdot\|$ the Euclidean norm, and by $\ext(K)$ we denote the set of extreme points of a convex set $K$. In the case $K$ is a polytope, $\ext(K)$ is the set of its vertices.

\begin{lemma}\label{lem:cov radius}
	Suppose a finite subset $A\subset\S^{n-1}$ is such that the interior of $K:=\conv(A)$ contains the origin. Then the covering radius of $A$ equals $\arccos((\max\{\|x\|:x\in \ext(K^\circ)\})^{-1})$.
\end{lemma}
\begin{proof}
	Since $K^\circ$ is a convex polytope and $x\mapsto \|x\|$ is a convex function, by Krein-Milman theorem, $\max\{\|x\|:x\in \ext(K^\circ)\} = \max\{\|x\|:x\in K^\circ\}=:\lambda$. A point $x\in\S^{n-1}$ is not covered by the union of spherical balls of radii $r$ centered at the points of $A$ if and only if $\langle x,y \rangle < \cos r$ for any $y\in A$, i.e. $(\cos r)^{-1}x$ lies in the interior of $K^\circ$. The covering radius of $A$ is then
	$
		\sup\{r>0: (\cos r)^{-1}\S^{n-1} \cap K^\circ \ne \emptyset \}=\lambda,
	$
	and the claim of the lemma follows.
\end{proof}

Under the hypothesis of the lemma, $K^\circ$ is a convex polytope given as the intersection of half-spaces. Therefore, the covering radius of $A$ can be efficiently computed after the half-space representation of $K^\circ$ is converted into the vertex representation. A function performing such a conversion is readily available in most softwares for mathematical computations, e.g. MatLab or SageMath. 

If $A$ possesses certain symmetries, then it may be possible to restrict the computations only to a certain part of the polytope. By $O(n)$ we denote the group of distance-preserving transformations of $\E^n$ that preserve origin. For $\T\subset O(n)$, the notation $\langle \T\rangle$ stands for the subgroup of $O(n)$ generated by $\T$.

\begin{lemma}
	Suppose $\T\subset O(n)$ is finite and $C\subset\E^n$ is such that $\bigcup\{T(C):T\in\langle \T\rangle \}=\E^n$. Further, suppose that for a finite subset $A\subset\S^{n-1}$ the interior of $K:=\conv(A)$ contains the origin and $T(A)=A$ for any $T\in\T$. Then the covering radius of $A$ equals $\arccos((\max\{\|x\|:x\in C\cap \ext(K^\circ)\})^{-1})$.
\end{lemma}
\begin{proof}
	Clearly, under the hypotheses of the lemma, $T(K^\circ)=K^\circ$ and $T(\ext(K^\circ))=\ext(K^\circ)$ for any $T\in \langle \T\rangle$. By \cref{lem:cov radius}, the covering radius of $A$ equals $\arccos((\max\{\|x\|:x\in \ext(K^\circ)\})^{-1})$. Suppose the maximum is attained at a point $x_0 \in \ext(K^\circ)$. Since $\bigcup\{T(C):T\in\langle \T\rangle \}=\E^n$, there exists $T\in \langle \T\rangle$ such that $x_0\in T(C)$, then $T^{-1}x_0\in C\cap \ext(K^\circ)$. We have $\|x_0\|=\|T^{-1}x_0\|\le \max\{\|x\|:x\in C\cap \ext(K^\circ)\} \le \max\{\|x\|:x\in \ext(K^\circ)\}= \|x_0\|$, and the lemma is proved.
\end{proof}

For example, if $A$ is origin-symmetric and invariant under permutations of coordinates, we can take $C=\{x\in\E^n: x_1\ge0, x_1\ge x_2\ge\dots\ge x_n \}$, which is applicable to all the cases in the next section. As $C$ is given as intersection of half-spaces, $C\cap K^\circ$ is a convex polytope, and $\max\{\|x\|:x\in C\cap \ext(K^\circ)\}=\max\{\|x\|:x\in \ext(C\cap K^\circ)\}$. 

%[[remark: tracking where the max is occurred gives a hole]]

\section{Proof of~\cref{thm:main}}\label{sec:proof}

%For a fixed dimension $n$ and $1\le m\le n$, by $A_{m}$ we denote the set of all possible $2^m \binom{n}{m}$ points from $\S^{n-1}$ which have exactly $m$ non-zero coordinates equal to $\pm\frac1{\sqrt{m}}$ each. By $B_m^*(a,b)$ we denote the set of all $\binom{n}{m}$ points from $\S^{n-1}$ whose $m$ coordinates are equal to $a$ and the remaining $n-m$ coordinates are equal to $b$, and then set $B_m(a,b):=B_m^*(a,b)\cup(-B_m^*(a,b))$. Let $\mu(x)$ denote the number of negative coordinates in $x\in\E^n$. For a set $J$ of positive integers, we define $A_{m}(J):=\{x\in A_{m}: \mu(x)\in J \}$. In particular, $A_{m}(2\Z_+)$ denotes all points from $A_{m}$ having an even number of negative coordinates.

We write $(x_1^{n_1},x_2^{n_2},\dots)$ to denote a vector that has some $n_i$ coordinates equal to $x_i$; for example, $(2,2,-1,0,\dots,0)\in\E^n$ can be written as $(2^2,-1,0^{n-3})$. For each $n$, $5\le n\le 15$, we construct an appropriate system of points $A\subset\S^{n-1}$ so that \cref{lem:reduction} is applicable. The set $A$ is obtained by taking all possible permutations of coordinates and symmetries about the origin of a certain smaller generating set of vectors. For convenience, we list the vectors of the generating set on a sphere which is not necessarily unit; the generating set can be normalized using a scalar multiple. Covering radii are found on a computer using the techniques of \cref{sec:numerical computation} and the code supplied in~\cite{self-pre-print}*{Appendix}. The constructions and the results are given in the table below ($|A|$ denotes the cardinality of $A$), where the decimal approximations are stated with the precision of 5 digits, while actual computational precision is double floating point arithmetic. 

For the dimensions $5\le n\le 10$, we computed the precise values of the covering radius by using exact computations in the field of rational numbers or in appropriate quadratic fields. Note that for $n=5$ the covering radius is equal to the one required by \cref{lem:cov radius}. All coordinates in our constructions are given by algebraic numbers, so with appropriate computational resources, the covering radii can be computed precisely. 

The running time of our script is well under a minute on a modern personal computer even for the case $n=15$ if floating point arithmetics is used. Getting precise results through symbolic computations takes longer for $n=9$ (five minutes) and $n=10$ (an hour).

\bgroup
\footnotesize%\small
\def\arraystretch{1.6}
\begin{center}
\begin{tabular}{|c|c|c|c|c|}
	\hline
	$n$ & vectors of generating set & $|A|$ & covering radius & $\arccos\sqrt{\tfrac{n-1}{2n}}\approx$ \\
	\hline
	$5$ & $(2^2,0^3)$, $(2,-2,0^3)$, $(2,(-1)^4)$ & $30$ & $\arccos\sqrt{\tfrac25}\approx 0.88608$ & $0.88608$ \\
	\hline
	$6$ & $(\sqrt{6},0^5)$, $(1^6)$, $(1^4,(-1)^2)$ & $44$ & $\arccos\frac23\approx0.84107$ & $0.86912$ \\
	\hline
	$7$ &  
	 \begin{tabular}{c}$(17^2,(-1)^5)$, $(13^2,(-7)^5)$,%\\
	 	$(23,(-3)^6)$, $(17,7^6)$\end{tabular} 
	& $112$ & $\arccos\tfrac{593}{55\sqrt{265}}\approx0.84688$ & $0.85707$ \\
	\hline	
	$8$ & \begin{tabular}{c} $(2^2,0^6)$, $(2,-2,0^6)$,%\\
		$(1^8)$, $(1^6,(-1)^2)$, $(1^4,(-1)^4)$ 
	\end{tabular} & $240$ & $\frac{\pi}{4}\approx0.78540$ & $0.84806$ \\
	\hline
	$9$ & \begin{tabular}{c} $(3^2,0^7)$, $(3,-3,0^7)$,\\
		$((\sqrt{2})^9)$, $((\sqrt{2})^7,(-\sqrt{2})^2)$, $((\sqrt{2})^5,(-\sqrt{2})^4)$ 
	\end{tabular} & $470$ & $\arccos\frac1{\sqrt{19-12\sqrt{2}}}\approx0.79265$ & $0.84107$ \\
	\hline
	$10$ & \begin{tabular}{c} $((\sqrt{10})^2,0^8)$, $(\sqrt{10},-\sqrt{10},0^8)$,\\
		$((\sqrt{2})^{10})$, $((\sqrt{2})^8,(-\sqrt{2})^2)$, $((\sqrt{2})^6,(-\sqrt{2})^4)$ 
	\end{tabular} & $692$ & $\arccos\frac1{\sqrt{20-8\sqrt{5}}}\approx0.81180$ & $0.83548$ \\
	\hline
	$11$ & \begin{tabular}{c} $(\sqrt{33},0^{10})$, $((\sqrt{11})^3,0^8)$,$((\sqrt{11})^2,-\sqrt{11},0^8)$,\\
		$((\sqrt{3})^{10},-\sqrt{3})$, $((\sqrt{3})^7,(-\sqrt{3})^4)$ 
	\end{tabular}  & $2024$ & $\approx0.82071$ & $0.83092$ \\
	\hline
	$12$ & \begin{tabular}{c} $(2\sqrt{3},0^{11})$, $(2^3,0^9)$,$(2^2,-2,0^9)$,\\
		$(1^{12})$, $(1^{10},(-1)^{2})$, $(1^8,(-1)^4)$, $(1^6,(-1)^6)$  
	\end{tabular} & $3832$ & $\approx0.78540$ & $0.82711$ \\
	\hline
	$13$ &  \begin{tabular}{c} $(\sqrt{39},0^{12})$, $((\sqrt{13})^3,0^{10})$,$((\sqrt{13})^2,-\sqrt{13},0^{10})$,\\
		$((\sqrt{3})^{13})$, $((\sqrt{3})^{12},-\sqrt{3})$, $((\sqrt{3})^{11},(-\sqrt{3})^2)$, \\ $((\sqrt{3})^{10},(-\sqrt{3})^3)$, $((\sqrt{3})^{9},(-\sqrt{3})^4)$, $((\sqrt{3})^{8},(-\sqrt{3})^5)$  
	\end{tabular} & $7074$ & $\approx0.79098$ & $0.82390$ \\
	\hline
	$14$ & \begin{tabular}{c} $(\sqrt{42},0^{13})$, $((\sqrt{14})^3,0^{11})$,$((\sqrt{14})^2,-\sqrt{14},0^{11})$,\\
		$((\sqrt{3})^{14})$, $((\sqrt{3})^{12},(-\sqrt{3})^2)$,\\ $((\sqrt{3})^{10},(-\sqrt{3})^4)$,  $((\sqrt{3})^{8},(-\sqrt{3})^6)$
	\end{tabular} & $11132$ & $\approx0.80395$ & $0.82114$ \\
	\hline
	$15$ & \begin{tabular}{c} $(2\sqrt{15},0^{14})$, $((2\sqrt{15})^2,-2\sqrt{15},0^{12})$, $((\sqrt{15})^4,0^{11})$, \\
		$(2^{15})$, $(2^{14},-2)$, $(2^{12},(-2)^3)$, $(2^{9},(-2)^6)$ \end{tabular}  & $16442$ & $\approx0.81793$ & $0.81876$ \\
\hline
\end{tabular}
\normalsize
\vskip2mm
Table~1. Constructions and covering radii
\end{center}
\egroup


\begin{bibsection}
\begin{biblist}
	
	
\bib{Be-Kh}{article}{
	author={Bezdek, K\'{a}roly},
	author={Khan, Muhammad A.},
	title={The geometry of homothetic covering and illumination},
	conference={
		title={Discrete geometry and symmetry},
	},
	book={
		series={Springer Proc. Math. Stat.},
		volume={234},
		publisher={Springer, Cham},
	},
	date={2018},
	pages={1--30},
	%	review={\MR{3816868}},
}

	
	\bib{Be-Ki}{article}{
		author={Bezdek, K.},
		author={Kiss, Gy.},
		title={On the X-ray number of almost smooth convex bodies and of convex
			bodies of constant width},
		journal={Canad. Math. Bull.},
		volume={52},
		date={2009},
		number={3},
		pages={342--348},
%		issn={0008-4395},
%		review={\MR{2547800}},
%		doi={10.4153/CMB-2009-037-0},
	}


\bib{self-pre-print}{article}{
	author={A. Bondarenko},
	author={A. Prymak},
	author={D. Radchenko},
	title={Spherical coverings and X-raying convex bodies of constant width},
	date={Dec. 2020},
	eprint={https://arxiv.org/abs/2011.06398v2}
}


\bib{Bo-Wi}{article}{
	author={B\"{o}r\"{o}czky, K\'{a}roly, Jr.},
	author={Wintsche, Gergely},
	title={Covering the sphere by equal spherical balls},
	conference={
		title={Discrete and computational geometry},
	},
	book={
		series={Algorithms Combin.},
		volume={25},
		publisher={Springer, Berlin},
	},
	date={2003},
	pages={235--251},
%	review={\MR{2038476}},
%	doi={10.1007/978-3-642-55566-4_10},
}


\bib{Co-Sl}{book}{
	author={Conway, J. H.},
	author={Sloane, N. J. A.},
	title={Sphere packings, lattices and groups},
	series={Grundlehren der Mathematischen Wissenschaften [Fundamental
		Principles of Mathematical Sciences]},
	volume={290},
	edition={3},
%	note={With additional contributions by E. Bannai, R. E. Borcherds, J.
%		Leech, S. P. Norton, A. M. Odlyzko, R. A. Parker, L. Queen and B. B.
%		Venkov},
	publisher={Springer-Verlag, New York},
	date={1999},
	pages={lxxiv+703},
%	isbn={0-387-98585-9},
%	review={\MR{1662447}},
%	doi={10.1007/978-1-4757-6568-7},
}

\bib{Da-Pr}{article}{
	author={F. Dai},	
	author={A. Prymak},	
	title={On directional Whitney inequality},
	journal={Canad. J. Math.},
	year={2021},
	pages={1--25},
	doi={10.4153/S0008414X21000110}
}

\bib{Da}{article}{
	author={Danzer, L.},
	author={Gr\"{u}nbaum, B.},
	title={\"{U}ber zwei Probleme bez\"{u}glich konvexer K\"{o}rper von P. Erd\H{o}s und von
		V. L. Klee},
	language={German},
	journal={Math. Z.},
	volume={79},
	date={1962},
	pages={95--99},
%	issn={0025-5874},
%	review={\MR{138040}},
%	doi={10.1007/BF01193107},
}


\bib{sagemath}{manual}{
	author={Developers, The~Sage},
	title={{S}agemath, the {S}age {M}athematics {S}oftware {S}ystem
		({V}ersion 9.2)},
	date={2020},
	note={{\tt https://www.sagemath.org}},
}



\bib{Du}{article}{
	author={Dumer, Ilya},
	title={Covering spheres with spheres},
	journal={Discrete Comput. Geom.},
	volume={38},
	date={2007},
	number={4},
	pages={665--679},
%	issn={0179-5376},
%	review={\MR{2365829}},
%	doi={10.1007/s00454-007-9000-7},
}


	\bib{HSTV}{article}{
author={Han Huang},	
author={Boaz A. Slomka},	
author={Tomasz Tkocz},	
author={Beatrice-Helen Vritsiou},	
title={Improved bounds for Hadwiger's covering problem via thin-shell estimates},
eprint={http://arxiv.org/abs/1811.12548}
}

\bib{constwidthbook}{book}{
	title={Bodies of Constant Width: An Introduction to Convex Geometry with Applications},
	author={Horst Martini},
	author={Luis Montejano},
	author={D\'eborah Oliveros},
	pages={486},
	date={2019},
	publisher={Springer International Publishing}
}


\bib{Na}{article}{
	author={Nasz\'{o}di, M\'{a}rton},
	title={On some covering problems in geometry},
	journal={Proc. Amer. Math. Soc.},
	volume={144},
	date={2016},
	number={8},
	pages={3555--3562},
%	issn={0002-9939},
%	review={\MR{3503722}},
%	doi={10.1090/proc/12992},
}


\bib{Pr-Sh}{article}{
	author={Prymak, A.},
	author={Shepelska, V.},
	title={On the Hadwiger covering problem in low dimensions},
	journal={J. Geom.},
	volume={111},
	date={2020},
	number={3},
	pages={42},
%	issn={0047-2468},
%	review={\MR{4153949}},
%	doi={10.1007/s00022-020-00554-3},
}
	
	
\bib{Ro1}{article}{
	author={Rogers, C. A.},
	title={A note on coverings},
	journal={Mathematika},
	volume={4},
	date={1957},
	pages={1--6},
%	issn={0025-5793},
%	review={\MR{90824}},
%	doi={10.1112/S0025579300001030},
}

\bib{Ro2}{article}{
	author={Rogers, C. A.},
	title={Covering a sphere with spheres},
	journal={Mathematika},
	volume={10},
	date={1963},
	pages={157--164},
%	issn={0025-5793},
%	review={\MR{166687}},
%	doi={10.1112/S0025579300004083},
}
	
	
	\bib{Sc}{article}{
		author={Schramm, Oded},
		title={Illuminating sets of constant width},
		journal={Mathematika},
		volume={35},
		date={1988},
		number={2},
		pages={180--189},
%		issn={0025-5793},
%		review={\MR{986627}},
%		doi={10.1112/S0025579300015175},
	}

\bib{Zong}{article}{
	author={Zong, Chuanming},
	title={A quantitative program for Hadwiger's covering conjecture},
	journal={Sci. China Math.},
	volume={53},
	date={2010},
	number={9},
	pages={2551--2560},
%	issn={1674-7283},
%	review={\MR{2718847}},
%	doi={10.1007/s11425-010-4087-3},
}

\end{biblist}
\end{bibsection}
\end{document}